\documentclass[11pt]{article}
\usepackage{amscd}
\usepackage{amssymb}
\usepackage{}
\usepackage{mathrsfs}
\usepackage{amsfonts}
\usepackage{amsfonts,amssymb,mathrsfs}
\usepackage{amsmath,amscd}
\usepackage[dvips]{graphicx}
\usepackage[all]{xy}
\usepackage{fancyhdr}
\usepackage{mathptmx,pslatex}
\usepackage{amsthm}
\usepackage[]{graphicx}

\newcommand{\defCategory}[2]{
  \newcommand{#1}{#2\defvariable}
  }

\newcommand{\defvariable}[2][]{
\if\relax\detokenize{#1}\relax  %if the first arg is empty
   \if\relax\detokenize{#2}\relax
    \else
    \left({#2}\right)
    \fi
\else
  ^{{#1}}\left({#2}\right)
\fi
 }

\begin{document}
\sloppy
\newcommand{\dickebox}{{\vrule height5pt width5pt depth0pt}}
\newtheorem{thm}{Theorem}[section]
\newtheorem{Def}{Definition}[section]
\newtheorem{rem}[thm]{Remark}
\newtheorem{Bsp}{Example}[section]
\newtheorem{Prop}[Def]{Proposition}
\newtheorem{Theo}[Def]{Theorem}
\newtheorem{Lem}[Def]{Lemma}
\newtheorem{Koro}[Def]{Corollary}

 \defCategory{\Dd}{\mathscr{D}}
 \defCategory{\Cc}{\mathscr{C}}
 \defCategory{\Kk}{\mathscr{K}}

\newcommand{\overpr}{$\hfill\square$}
\newcommand{\pd}{{\rm proj.dim\, }}
\newcommand{\id}{{\rm inj.dim\, }}
\newcommand{\fd}{{\rm fin.dim}\,}
\newcommand{\add}{{\rm add \, }}
\newcommand{\Hom}{{\rm Hom \, }}
\newcommand{\gldim}{{\rm gl.dim}\,}
\newcommand{\End}{{\rm End \, }}
\newcommand{\Ext}{{\rm Ext}}
\newcommand{\D}{\rm D \,}
\newcommand{\Coker}{{\rm Coker}\,\,}
\newcommand{\cpx}[1]{#1^{\bullet}}
\newcommand{\Dz}[1]{{\rm D}^+(#1)}
\newcommand{\Df}[1]{{\rm D}^-(#1)}
\newcommand{\DFz}[1]{{\rm  {\mathscr{D}_{F}}^+}(#1)}
\newcommand{\DFf}[1]{{\rm  {\mathscr{D}_{F}}^-}(#1)}
 \def\Db#1{\Dd[{\rm b}]{#1}}
\newcommand{\DFb}[1]{{\rm {\mathscr{D}_{F}^b}}(#1)}
\newcommand{\DF}[1]{{\rm {\mathscr{D}_{F}}}(#1)}

\newcommand{\C}[1]{{\rm C}(#1)}
\newcommand{\CFa}[1]{{\rm C_{F}^{ac}}(#1)}
\newcommand{\Cz}[1]{{\rm C}^+(#1)}
\newcommand{\Cf}[1]{{\rm C}^-(#1)}
\newcommand{\Cb}[1]{{\rm C^b}(#1)}
\newcommand{\K}[1]{{\rm K}(#1)}
\newcommand{\KF}[1]{{\rm \mathscr{K}_{F}}(#1)}
\newcommand{\Ka}[1]{{\rm \mathscr{K}^{ac}}(#1)}
\newcommand{\KFa}[1]{{\rm \mathscr{K}_{F}^{ac}}(#1)}
\newcommand{\Kz}[1]{{\rm \mathscr{K}}^+(#1)}
\newcommand{\Kbz}[1]{{\rm \mathscr{K}}^{-,b}(#1)}
\newcommand{\Kbf}[1]{{\rm \mathscr{K}}^{+,b}(#1)}

\def\rank{{\rm rank}}
\def\KFb#1{\Kk[{\rm -,Fb}]{#1}}
\def\Kf#1{\Kk[{\rm -}]{#1}}
 \def\Kb#1{\Kk[{\rm b}]{#1}}
\newcommand{\modcat}[1]{#1\mbox{{\rm -mod}}}
\newcommand{\smod}[1]{#1\mbox{{\rm -\underline{mod}}}}
\newcommand{\opp}{^{\rm op}}
\newcommand{\otimesL}{\otimes^{\rm\bf L}}
\newcommand{\rHom}{{\rm\bf R}{\rm Hom}\,}
\newcommand{\projdim}{\pd}
\newcommand{\lra}{\longrightarrow}
\newcommand{\ra}{\rightarrow}
\newcommand{\rad}{{\rm rad \, }}
\newcommand{\Lra}{\Longrightarrow}
\newcommand{\tra}{\twoheadrightarrow}
\newcommand{\Htp}{{\rm Htp}}
\newcommand{\raf}[1]{\stackrel{#1}{\ra}}
\newcommand{\lraf}[1]{\stackrel{#1}{\lra}}
\newcommand{\proja}{\mathcal{P}_{\!\!\!\scriptscriptstyle{\mathcal{A}}}}
\newcommand{\projb}{\mathcal{P}_{\!\!\scriptscriptstyle{\mathcal{B}}}}
\newcommand{\projc}{\mathcal{P}_{\!\!\scriptscriptstyle{\mathcal{C}}}}
 \def\con{{\rm con}}
\newcommand{\sta}{\underline{\mathcal{A}}}
\newcommand{\stb}{\underline{\mathcal{B}}}
\newcommand{\phd}{{\rm \phi\mbox{-}dim\, }}
\newcommand{\Ima}{{\rm Ima}}
\newcommand{\Ker}{{\rm Ker}}
\newcommand{\pmodcat}[1]{#1\mbox{{\rm -proj}}}
\newcommand{\imodcat}[1]{#1\mbox{{\rm -inj}}}
\def\Modcat#1{#1\mbox{-}{\rm Mod}}
\def\modcat#1{#1\mbox{-}{\rm mod}}
\def\stModcat#1{#1\mbox{-}\underline{{\rm Mod}}}
\def\stmodcat#1{#1\mbox{-}\underline{{\rm mod}}}
\def\pModcat#1{#1\mbox{-}{\rm Proj}}
\def\pmodcat#1{#1\mbox{-}{\rm proj}}
\newcommand{\Ob}{{\rm Obj}}

\def\bfL{\mathbf{L}}
\def\calA{{\mathcal A}}
\def\calB{{\mathcal B}}
\def\calC{{\mathcal C}}
\def\calD{{\mathcal D}}
\def\Ga{\hbox{$\mathitB$}}
\def\La{\hbox{$\mathitA$}}
\newcommand{\extdm}{{\rm ext.dim }}

\def\GProj#1{#1\mbox{-}{\cal GP}}
\def\GInj#1{#1\mbox{-}{\cal GI}}
\def\Gproj#1{#1\mbox{-}{\it f\cal GP}}
\def\stGProj#1{#1\mbox{-}{\underline{{\cal GP}}}}
\def\stGproj#1{#1\mbox{-}\underline{{\it f\cal GP}}}
\def\FstGProj#1{#1\mbox{-}{\underline{{\cal FGP}}}}

\newcommand{\stHom}{\underline{\Hom}}
\newcommand{\perpgg}{^{\perp_{\gg 0}}}
\newcommand{\perpg}[1]{^{\perp_{>{#1}}}}
\newcommand{\wrd}{\rm w.resl.dim}

{\Large \bf
\begin{center}  Relative derived equivalences and relative
Igusa-Todorov dimensions 
 \end{center}}

\medskip

\centerline{\sc Peizheng Guo, Shengyong Pan}

\begin{center}  School of Mathematics and Statistics,\\
 Beijing Jiaotong University,  Beijing 100044,\\
People's Republic of China\\ E-mail:shypan@bjtu.edu.cn
\end{center}

\medskip
\abstract{ 
Let $A$ be an Artin algebra and $F$ a non-zero subfunctor of $\Ext_A^{1}(-,-)$. In this paper, we characterize the relative $\phi$-dimension of $A$ by the bi-functor $\Ext_F^1(-,-)$. Furthermore, we show that the finiteness of relative 
$\phi$-dimension of an Artin algebra is invariant under relative derived equivalence. More precisely, for an Artin algebra $A$, assume that $F$ has enough projectives and injectives, such that
there exists $G\in \modcat{A}$ such that $\add G=\mathcal {P}(F)$, where $\mathcal {P}(F)$ is the category of all $F$-projecitve $A$-modules. If $\cpx{T}$ is a relative tilting complex over $A$ with term length $t(\cpx{T})$ such that $B=\End(\cpx{T})$, 
then we have
$\phd_{F}(A)-t(T^{\bullet})\leq \phd(B)\leq
\phd_{F}(A)+t(T^{\bullet})+2$.}

\medskip {\small {\it 2020 AMS Classification}: 18E30,16G10;16S10,18G15.

}

\section{Introduction}
As is known, Hochschild \cite{Hoch} introduced relative homological algebra in
categories of modules. Heller, Butler and Horrocks
developed it in more general categories with a relative abelian
structure. After that Auslander and Solberg \cite{ASo1,ASo2,ASo3,ASo4,ASo5}
applied relative homological algebra to the representation theory of
Artin algebra, they studied relative homology in terms of
subbifunctors of the functor $\Ext^1(-,-)$ and developed the general
theory of relative cotilting modules for Artin algebras.

Derived categories were invented by Grothendieck-Verdier \cite{Ver}
in the early sixties. Today, they have widely been used in many
branches: algebraic geometry, stable homotopy theory, representation
theory, etc. In representation theory of algebras, it is of
interest to investigate whether two algebras have equivalent
derived categories. Rickard \cite{Ri1} used the theory of tilting complexes to give a condition for when the derived categories of two rings are equivalent.
According to his
theorem, two rings $A$ and $B$ are derived equivalent if
and only if there is a tilting complex for $A$ such that $B$ is
isomorphic to the endomorphism algebra of this complex.  
Derived categories have been used effectively in relative
homological algebra. The main idea of relative homological algebra
is to replace projective modules by relative projective modules. It
is natural to study the corresponding version of the derived
category in this context. Since then, the relative derived
categories and relative tilting theory of Artin algebras have been
extensively studied. Gao and Zhang \cite{GZ} used
Gorenstein homological algebra to get Gorenstein derived categories.
Buan \cite{Bu1} also studied relative derived categories by
localizing relative quasi-isomorphisms. Using the notion of relative
derived categories, he generalized Happel's result on derived
equivalences induced by tilting module to the relative setting.
Motivated by Buan's work, Pan \cite{P1} introduced relative derived equivalences
for Artin algebras by the theory of relative tilting complexes, and some invariance of relative derived equivalences is founded. 

The aim of this paper is to discuss the
relationships between relative derived equivalences and relative
Igusa-Todorov dimensions which was defined in \cite{LM}.
To describe the main result, it is convenient to fix some notations.
Let $A$ and $B$ be Artin algebras. Assume that $F$ is a
subbifunctor of $\End^{1}_{A}(-,-)$ which is of finite type.
Denote by $\phd(A)$ the Igusa-Todorov dimensions of $A$. Let $n\geq0$
be an integer. If the complex $\cpx{X}$ has the form:
$$
\cdots\ra X^{-n}\ra X^{-n+1}\ra\cdots\ra X^{-1}\ra X^{0}\ra\cdots,
$$
with $X^{i}\neq0$ and the differential being radical map for $-n\leq
i \leq0$, then $n$ is called the term length of $\cpx{X}$, denoted
by $t(\cpx{X})$.

Our main result can be stated as follows:

\noindent{\bf Theorem} (see Theorem \ref{Main-result}) {\it  Let $L:\DFb{A}\ra \Db{B}$ be a
relative derived equivalence. Suppose $T^{\bullet}$ is the relative
tilting complex associated to $L$. Then we have
$\phd_{F}(A)-t(T^{\bullet})\leq \phd(B)\leq
\phd_{F}(A)+t(T^{\bullet})+2$. }

We give the upper and lower bounds of $\phd(B)$ in the term length of relative tilting complex and
$\phd_{F}(A)$. In this theorem,
if $F=\Ext_{A}^{1}(-,-)$, then $T^{\bullet}$ is a tilting
complex for $A$ such that $\End(\cpx{T})\simeq B$. In the
proof of this theorem, we infer that $\phd(A)$ and $\phd(B)$ satisfy a
similar formula but not the same as in this theorem, namely
 $\phd(A)-t(T^{\bullet})\leq \phd(B)\leq
\phd(A)+t(T^{\bullet})$. This is the main result in 
\cite[Theorem 4.10]{FLH}.

This paper is organized as follows. In Section 2, we consider closed
subbifunctors $F$ of $\Ext_{\mathscr{A}}^{1}(-,-)$, where
$\mathscr{A}$ is a small abelian category. Then we recall the notion
of the relative derived category of $\mathscr{A}$ which was defined
in \cite{Bu1}. We recall the relative derived equivalences for Artin algebras in Section 3. 
In Section 4, we prove the
main result.

%%%%%%%%%%%%%%%%%%%%%%%%%%%%%%%%%%%%%%%%%%%%%%%%%%%%%%%%%%%%%%%%%%%%%%%%
\section{Relative derived categories for abelian categories}
Let us explain the notion of relative derived categories. The notion
of relative derived categories was introduced earlier by Generalov
\cite {Ge}.

Let $\mathscr{A}$ be an abelian category. Suppose $A,
C\in\mathscr{A}$. Denote by $\Ext^{1}_{\mathscr{A}}(C,A)$ the set of
all exact sequences $0\ra A \ra B \ra C \ra 0 $ in $\mathscr{A}$
modulo the equivalence relation which is defined in the following
usual way. Two exact sequences are equivalent if the following
commutative diagram is commutative.
$$\xymatrix{
0\ar[r]&A\ar[r]\ar@{=}[d]   &    B\ar[r]\ar[d]    &    C\ar[r]\ar@{=}[d]& 0\\
0\ar[r]&A\ar[r]         &    B'\ar[r]          & C\ar[r]& 0}$$

Since that any additive category has finite direct sums and in
particular uniquely defined diagonal and codiagonal maps and that an
abelian category has pullback and pushout pairs, it follows that
$\Ext^{1}_{\mathscr{A}}(C,A)$ becomes an abelian group under Baer
sum. Therefore, $\Ext^{1}_{\mathscr{A}}(-,-)$ defines an additive
bifunctor $\mathscr{A}^{op} \times \mathscr{A} \lra \textbf{Ab}$,
where $\textbf{Ab}$ is the category of abelian groups.

Consider additive non-zero subbifunctors $F$ of $\Ext^{1}(-,-)$. To
each subfunctor corresponds a class of short exact sequences which
are called $F$-exact sequences. The class of $F$-exact sequences is
closed under the operations of pushout, pullback, Baer sums and
direct sums (see \cite{ASo1} or \cite{DRSS}). Given a subbifunctor
$F$ of $\Ext^{1}(-,-)$, we say that an exact sequence
$$\eta:0\ra X \ra Y \ra Z \ra 0 $$ in $\mathscr{A}$ is an $F$-exact
sequence if $\eta$ is in $F(Z,X)$. If $ 0\ra X \stackrel{f}\ra Y
\stackrel{g}\ra Z\ra 0 $ is an $F$-exact sequence, then $f$ is
called an $F$-monomorphism and $g$ is called an $F$-epimorphism. An
object $P$ is said to be $F$-projective if for each $F$-exact
sequence $0\ra X \ra Y \ra Z \ra 0 $, the sequence
$$0\ra \mathscr{A}(P,X) \ra \mathscr{A}(P,Y) \ra \mathscr{A}(P,Z)\ra
0$$ is exact. An object $I$ is called $F$-injective if for each
$F$-exact sequence $0\ra X \ra Y \ra Z \ra 0 $, the sequence $$0\ra
\mathscr{A}(X,I) \ra \mathscr{A}(Y,I) \ra \mathscr{A}(Z,I)\ra 0$$ is
exact. The subcategory of $\mathscr{A}$ consisting of all
$F$-projective (resp. $F$-injective) modules is denoted by
$\mathcal{P}(F)$ (resp. $\mathcal{I}(F)$).

We have the following characterization of when subbifunctors of
$\Ext^{1}(-,-)$ have enough projectives or injectives.

\begin{Lem} \label{2.1}$\rm\cite[Theorem\,1.12]{ASo1}$ Let $F$ be a subbifunctor of $\Ext^{1}(-,-)$.

$(1)$ $F$ has enough projectives if and only if $F=F_{\mathcal
{P}(F)}$ and $\mathcal {P}(F)$ is contravariantly finite in
$A$-mod.

$(2)$ $F$ has enough injectives if and only if $F=F^{\mathcal
{I}(F)}$ and $\mathcal {I}(F)$ is covariantly finite in
$A$-mod.

$(3)$ If there is a finite number of indecomposable relative
projectives (injectives) up to isomorphism, then there is also a
finite number of relative injectives (projectives), and these
numbers are the same.
\end{Lem}

Recall from \cite{Bu1, DRSS} that an additive subbifunctor $F$ is
said to be closed if the following equivalence statements hold.

1) The composition of $F$-epimorphisms is an $F$-epimorphism.

2) The composition of $F$-monomorphisms is an $F$-monomorphism.

3) For each object $X$ the functor $F(X,-)$ is half exact on
$F$-exact sequences.

4) For each object $X$ the functor $F(-,X)$ is half exact on
$F$-exact sequences.

5) The category $\mathscr{A}$ with respect to the $F$-exact
sequences is an exact category.

We will give some basic examples of closed subbifunctors. Let
$\mathcal {X}$ be a full subcategory of $\mathscr{A}$ and for each
pair of objects $A$ and $C$ in $\mathscr{A}$, we define
$$
F_{\mathcal {X}}(C,A)=\{0\ra A \ra B\ra C\ra 0\mid (\mathcal
{X},B)\ra(\mathcal {X},C)\ra 0 \text{\;is exact} \}.
$$
Dually we define for each pair of objects $A$ and $C$ in
$\mathscr{A}$
$$
F^{\mathcal {X}}(C,A)=\{0\ra A \ra B\ra C\ra 0\mid (B,\mathcal
{X})\ra(A,\mathcal {X})\ra 0 \text{\;is exact} \}.
$$

\begin{Lem}\label{2.2} $\rm\cite[Proposition\,1.7]{DRSS}$
The additive subbifunctors $F_{\mathcal {X}}$ and $F^{\mathcal {X}}$
of $\Ext^{1}_{\mathscr{A}}(-,-)$ are closed for any subcategory
$\mathcal {X}$ of $\mathscr{A}$.
\end{Lem}

\noindent{\bf Remark.} By Lemma \ref{2.1} and Lemma \ref{2.2}, it
follows that if a subbifunctor $F$ has enough projectives or
injectives, then $F$ is a closed subfunctor.

Let $\K{\mathscr{A}}$ be the homotopy category of complexes over
$\mathscr{A}$. We denote by $[1]$ the shift functor. A complex
$X^{\bullet}$ with differential $d_{X}^{\bullet}$ is said to be
$F$-acyclic if for each $i$ the induced complex $$0\ra \Ima( d^{i-1}) \ra
X^{i} \ra \Ima (d^{i}) \ra 0 $$ is an $F$-exact sequence. A map $h$ in
$\K{\mathscr{A}}$ is called an $F$-quasi-isomorphism if the mapping
cone $M(h)$ is an $F$-acyclic complex. If $F$ has enough
projectives, then $F$-exact sequences are exact sequences and
$F$-acyclic complexes are acyclic complexes. But the converse is not
true.

If the class of $F$-acyclic complexes are closed under the operation
of mapping cone, then we call $F$ a triangulated subbifunctor. In
this case the class of $F$-acyclic complexes is a null-system in
$\K{\mathscr{A}}$. Since the class of $F$-quasi-isomorphisms is a
multiplicative system, it follows that we localize with respect to
this system.

\begin{Lem} $\rm\cite[Theorem\,2.4]{Bu1}$
A subbifunctor is triangulated if and only if it is closed.
\end{Lem}

Now we assume that $F$ is a closed subbifunctor. It is clear that
$$\mathscr{N}=\{ X^{\bullet}\in \text{Obj}(\K{\mathscr{A}})\mid
X^{\bullet} \text{\;is an F-acyclic complex} \}$$ is a null system
and a thick subcategory of $\K{\mathscr{A}}$. Define a morphism set
\begin{eqnarray*}
\Sigma(\mathscr{N})=\{X^{\bullet} \stackrel{f^{\bullet}}\ra
Y^{\bullet} \mid\text{such that\;} X^{\bullet}
\stackrel{f^{\bullet}}\ra Y^{\bullet} \ra Z^{\bullet} \ra
X^{\bullet}[1]\text{\; is a distinguished triangle
in\;}\\\K{\mathscr{A}} \text{\;with\;} Z^{\bullet} \in\mathscr{N}
\}.
\end{eqnarray*} The relative derived
category of $\mathscr{A}$ is defined to be the Verdier quotient,
that is,
$\DFb{\mathscr{A}}:=\K{\mathscr{A}}/\mathscr{N}=\Sigma(\mathscr{N})^{-1}\K{\mathscr{A}}$.

The objects of $\DFb{A}$ are the same as for
$\K{\mathscr{A}}$. A map in $\DFb{A}$ from $X^{\bullet}$ to
$Y^{\bullet}$ is the equivalence class of "roofs", that is, of
fractions $\cpx{g}/\cpx{f}$ of the form $\xymatrix{ \cpx{X}&
Z^{\bullet} \ar_{f^{\bullet}}[l]\ar^{g^{\bullet}}[r]& \cpx{Y}}$,
where $Z^{\bullet}\in \K{\mathscr{A}}$, $f^{\bullet}:Z^{\bullet}\ra
X^{\bullet} $ is an $F$-quasi-isomorphism, and
$g^{\bullet}:Z^{\bullet}\ra X^{\bullet}$ is a morphism in
$\K{\mathscr{A}}$. Two such roofs $\xymatrix{ \cpx{X}&  Z^{\bullet}
\ar_{f^{\bullet}}[l]\ar^{g^{\bullet}}[r]& \cpx{Y}}$ and $\xymatrix{
\cpx{X'}&  \cpx{Z'}\ar_{\cpx{f'}}[l]\ar^{\cpx{g'}}[r]& \cpx{Y'}}$
are equivalent if there exists a commutative diagram
$$\xymatrix{
   &  Z^{\bullet} \ar_{f^{\bullet}}[dl]\ar^{g^{\bullet}}[dr]        \\
X^{\bullet}   & W^{\bullet}\ar_{h^{\bullet}}[l]\ar[u]\ar[d]&    Y^{\bullet} , \\
&  Z'^{\bullet} \ar^{f'^{\bullet}}[ul]\ar_{g'^{\bullet}}[ur] }$$
where $h^{\bullet}$ is an $F$-quasi-isomorphism. Note that the
diagram of the form $\xymatrix{ \cpx{X}&  Z^{\bullet}
\ar_{f^{\bullet}}[l]\ar^{g^{\bullet}}[r]& \cpx{Y}}$ will be called a
left roof.

\section{Relative derived equivalences for Artin algebras}

Let $A$ be an arbitrary ring with identity. The category $\Modcat{A}$ of unitary left  $A$-modules is an abelian category with enough projective objects.  We use $\modcat{A}$ to denote the full subcategory of $\Modcat{A}$ consisting of  finitely presented $A$-modules, that is, $A$-modules $X$ admitting a projective presentation $P_1\ra P_0\lra X\ra 0$ with $P_i$ finitely generated projective for $i=0, 1$.  The category $\modcat{A}$ is abelian when $A$ is left coherent. The full subcategory of $\Modcat{A}$ consisting of all projective modules is denoted by $\pModcat{A}$, and the category of finitely generated projective $A$-modules is written as $\pmodcat{A}$. Note that $\pmodcat{A}$ are precisely those projective modules in $\modcat{A}$. The stable category of $\Modcat{A}$ is denoted by $\stModcat{A}$, in which morphism space is denoted by $\stHom_A(X, Y)$ for each pair of $A$-modules $X$ and $Y$.  For a full subcategory $\mathscr{X}$ of $\Modcat{A}$, we denote by $\underline{\mathscr{X}}$ the full subcategory of $\stModcat{A}$ consisting of all modules in $\mathscr{X}$. However, the full subcategory of $\stModcat{A}$ consisting of finitely presented modules is denoted by $\stmodcat{A}$

Suppose $A$ is an Artin $R$-algebras. Let $F$
be a non-zero subfunctor of $\Ext_{A}^{1}(-,-)$. Let
$A$-mod be the category of finitely generated left
$A$-modules. Denote by $\mathcal{P}(F)$ the subcategory of
$A$-mod consisting of all $F$-projective $A$-modules.
The subcategory $_{A}\mathcal{P}$ of finitely generated
projective $A$-modules is contained in $\mathcal{P}(F)$.
The relative stable category of $\Modcat{A}$ is denoted by $\stModcat{A}_F$, in which morphism space factors through $F$-projective $A$-modules is denoted by $\stHom_F(X, Y)$ for each pair of $A$-modules $X$ and $Y$. The full subcategory of $\stModcat{A}_F$ consisting of finitely presented modules is denoted by $\stmodcat{A}_F$
Assume that $F$ has enough projectives and injectives, such that
there exists $G\in A$-mod such that $\add G=\mathcal {P}(F)$.
Let $\DFb{A}$ be the relative bounded derived category of
$A$-mod.

\medskip
Two Artin algebra $A$ and $B$ are said to be {\em relatively derived equivalent} if the following equivalent conditions are satisfied.

\smallskip
$(1)$ $\DFb{\modcat{A}}$ and $\Db{\modcat{B}}$ are equivalent as
triangulated categories.

$(2)$ $\Kb{\mathcal {P}(F)}$ and $\Kb{_{B}\mathcal {P}}$ are
equivalent as triangulated categories.

$(3)$ $\Kf{\mathcal {P}(F)}$ and $\Kf{_{B}\mathcal {P}}$ are
equivalent as triangulated categories.

$(4)$ There is a complex $T^{\bullet}\in\Kb{\mathcal {P}(F)}$ satisfying

   \qquad $(i)$ $\Hom(T^{\bullet},T^{\bullet}[i])= 0$ for i $\neq$   
0.

   \qquad $(ii)$ $\add(T^{\bullet})$, the category of direct summands of
    finite direct sums of copies of $T^{\bullet}$, generates $\Kb{\mathcal {P}(F)}$ as
a triangulated category,

\quad\quad  such that the endomorphism algebra of $\cpx{T}$ in $\Kb{\pmodcat{A}}$ is isomorphic to $B$.

\medskip
{\parindent-0pt For the } proof that the above conditions are indeed equivalent, we refer to \cite{P1}. 

\medskip
{\parindent=0pt   A} complex $\cpx{T}$ satisfying the conditions (a) and (b) above is called a {\em relative tilting complex}. A triangle equivalence functor $L: \DFb{\Modcat{A}}\ra \Db{\Modcat{B}}$ is called a {\em relative derived equivalence}. In this case, the image $L(G)$ is isomorphic in $\Db{\Modcat{B}}$ to a tilting complex, and there is a relative tilting complex $\cpx{T}$ over $A$ such that $L(\cpx{T})$ is isomorphic to $B$ in $\Db{\Modcat{B}}$.  The complex $\cpx{T}$ is called an {\em associated $F$-tilting complex } of $L$.  The following is an easy lemma for the associated $F$-tilting complexes.  For the convenience of the reader, we provide a proof.

\begin{Lem}\label{relative-til}
Let $A$ and $B$ be two Artin algebras, and let $L: \DFb{\Modcat{A}}\lra \Db{\Modcat{B}}$ be a derived equivalence. Then $L(G)$ is isomorphic in $\Db{\modcat{B}}$ to a complex $\cpx{\bar{T}}\in\Kb{\pmodcat{B}}$  of the form
$$0\lra \bar{T}^0\lra \bar{T}^1\lra\cdots\lra\bar{T}^n\lra 0$$
for some $n\geq 0$ if and only if $L^{-1}(B)$ is isomorphic  in $\Db{\Modcat{A}}$ to a complex $\cpx{T}\in\Kb{\mathcal{P}(F)}$ of the form
$$0\lra T^{-n}\lra\cdots\lra T^{-1}\lra T^0\lra 0.$$
\label{lemma-tiltCompForm}
\end{Lem}

\begin{proof}
Considering the homological group of $\cpx{\bar{T}}$, we
have
$$
H^{i}(\bar{T}^{\bullet})\simeq
\Hom_{\Kb{B}}(B,\bar{T}^{\bullet}[i])\simeq
\Hom_{\Db{B}}(B,\bar{T}^{\bullet}[i])\simeq
\Hom_{\DFb{A}}(T^{\bullet},G[i]).
$$
Then we have
$\Hom_{\DFb{A}}(T^{\bullet},G[i])\simeq
\Hom_{\Kb{A}}(T^{\bullet},G[i])$. Then
$H^{i}(\bar{T}^{\bullet})=0$ for $i>n$ or $i<0$ by transferring
shifts. Since $\bar{T}^{\bullet}\in \Kb{_{B}\mathcal {P}}$, it
follows that $\bar{T}^{\bullet}$ has the following form
$$
\cdots\to 0\to \bar{T}^{-r}\to \bar{T}^{-r+1}\to \cdots \to
\bar{T}^{-1}\stackrel{d^{-1}}\to \bar{T}^{0} \stackrel{d^{0}}\to
\bar{T}^{1}\to \bar{T}^{s}\to0 \to\cdots
$$
where $r,s \in \mathbb{N}$. If $H^{i}(\bar{T}^{\bullet})=0$ for
$i>n$, then the following sequence $$0\to \Ima(d_T^{n})\to
\bar{T}^{n+1}\to \cdots \to \bar{T}^{s}\to0 \to\cdots$$ is a split
exact sequence. Consequently, $Im(d_T^{n})$ is a projective
$B$-module. Then the complex $\bar{T}^{\bullet}$ is isomorphic
in $\Db{B}$ to a complex of the form
$$0\to\bar{T}^{-r}\to \bar{T}^{-r+1}\to \cdots
\to \bar{T}^{-1}\to \bar{T}^{0}\to \cdots \to \bar{T}^{n-1}\to
\Ker(d_T^{n})\to 0,
$$where $\Ker(d_T^{n})$ is a projective $B$-module.
On the other hand, we have
$$
H^{i}(\Hom_{B}(\bar{T}^{\bullet},B))\simeq
\Hom_{\Db{B}}(\bar{T}^{\bullet}, B[i])\simeq
\Hom_{\DFb{A}}(G, T^{\bullet}[i])\simeq \Hom_{\Kb{A}}(G,
T^{\bullet}[i])=0.
$$for $i>0$.
It follows that $$0\to \Ima(d_T^{0},B)\to
(\bar{T}^{-1},B)\to (\bar{T}^{-2},B)\to\cdots$$ is split
exact in $\Kb{B^{op}}$, where
$\Ima(d_T^{0},B)=(\Ima(d_T^{-1}),B)$. We conclude that
$\Ker(d_T^{0},B)=(\Coker(d_T^{-1}),B)$ is a projective
$B$-module. Consequently, $\Coker(d_T^{-1})$ and $\Ima(d_T^{-1})$
are projective $B$-modules. Therefore, the complex
$\bar{T}^{\bullet}$ is isomorphic a complex of the following form
$$
\cdots \to 0 \to \Coker(d_T^{-1})\to \bar{T}^{1}\to \cdots \to
\bar{T}^{n-1}\to \Ker(d_T^{n})\to 0 \quad\text{in}\quad
\Kb{_{B}\mathcal {P}}.
$$

Conversely, suppose that the tilting complex $\cpx{\bar{T}}$ associated to the quasi-inverse of $L$ of the following form $
\cdots \to 0 \to\bar{T}^{0}\to \bar{T}^{1}\to \cdots \to
\bar{T}^{n-1}\to\bar{T}^{n}\to 0 \quad\text{in}\quad
\Kb{_{B}\mathcal {P}}.
$
Considering the homological group of $\cpx{T}$, we
have
$$
H^{i}(\Hom(P,T^{\bullet}))\simeq
\Hom_{\Kb{A}}(P,T^{\bullet}[i])\simeq
\Hom_{\DFb{A}}(P,T^{\bullet}[i])\simeq
\Hom_{\Db{B}}(\cpx{\bar{T}_P},B[i]).
$$
Then we have
$\Hom_{\Db{B}}(\bar{T}_{P}^{\bullet},B[i])\simeq
\Hom_{\Kb{A}}(\bar{T}_{P}^{\bullet},B[i])$. Then
$H^{i}(\Hom(P,T^{\bullet}))=0$ for $i>0$ or $i<-n$ by transferring
shifts. Since $T^{\bullet}\in \Kb{\mathcal {P}(F)}$, it
follows that $T^{\bullet}$ has the following form
$$
\cdots\to 0\to T^{-r}\to T^{-r+1}\to \cdots \to
T^{-1}\stackrel{d^{-1}}\to T^{0} \stackrel{d^{0}}\to
T^{1}\to T^{s}\to0 \to\cdots
$$
where $r,s \in \mathbb{N}$. If $H^{i}(P,T^{\bullet})=0$ for
$i>0$, then the following sequence $$
0\ra \Ima(\Hom(P,d_{T_{Y}}^{0}))\ra \Hom(P,T_{Y}^{1})\ra \cdots.
$$
is a split
exact sequence. 
We conclude that there exist two $F$-acyclic complexes:
$$
0\ra \Ima(d_{T_{Y}}^{0})\ra T_{Y}^{1}\ra \cdots.
$$
Consequently, $\Ima(d_T^{0})$ is a $F$-projective
$B$-module. Then the complex $\bar{T}^{\bullet}$ is isomorphic
in $\Db{B}$ to a complex of the form
$$0\to T^{-r}\to T^{-r+1}\to \cdots
\to T^{-1}\to
\Ker(d_T^{0})\to 0,
$$where $\Ker(d_T^{n})$ is a projective $B$-module.
On the other hand, we have
$$
H^{i}(\Hom_{A}(T^{\bullet},G))\simeq
\Hom_{\DFb{A}}(T^{\bullet}, B[i])\simeq
\Hom_{\Db{B}}(B, \bar{T}^{\bullet}[i])\simeq \Hom_{\Kb{A}}(G,
T^{\bullet}[i])=0.
$$for $i>n$.
It follows that $$0\to \Ima(d_T^{-n},G)\to
(\bar{T}^{-n-1},G)\to (\bar{T}^{-n-2},G)\to\cdots$$ is split
exact in $\Kb{\End(G)^{op}}$, where
$\Ima(d_T^{-n},B)=\Hom(\Ima(d_T^{-n-1}),G)$. We conclude that
$\Ker(d_T^{-n},G)=(\Coker(d_T^{-n-1}),G)$ is a projective
$\End(G)^{op}$-module. Consequently, $\Coker(d_T^{-n-1})$ and $\Ima(d_T^{-n-1})$
are $F$-projective $A$-modules. Therefore, the complex
$T^{\bullet}$ is isomorphic a complex of the following form
$$
\cdots \to 0 \to \Coker(d_T^{-1})\to T^{1}\to \cdots \to
T^{n-1}\to \Ker(d_T^{n})\to 0 \quad\text{in}\quad
\Kb{\mathcal {P}(F)}.
$$
This completes the proof.
\end{proof}

\section{Relative Igusa-Todorov dimension}

Let $K_F(A)$  the quotient abelian group generated by the set of iso-classes $\{[M]|M\in\modcat{A}\}$ modulo the relations $(a)$ $[N]-[X]-[Y]$ if $N\simeq X\oplus Y$ and $(b)$ $[P]$ if $P$ is $F$-projective. Then, $K_F(A)$ is the free abelian group generated by the iso-classes of finitely generated indecomposable non-$F$-projective $A$-modules. Note that every element of $K_F(A)$ is equal to $[M]-[N]$ for some $M,N\in\modcat{A}$.
Let $\langle M\rangle$ denote the $\mathbb{Z}$-submodule of $K_F(A)$ generated by the indecomposable non-$F$-projective direct summands of $M$.

\begin{Lem} (Fitting’s Lemma) Let $R$ be a noetherian ring. Consider a left $R$-module $M$
and $f\in\End_R(M)$. Then, for any finitely generated $R$-submodule $X$ of $M$, there is a nonnegative
integer 
$$
\eta_{f(X)}:=\min\{k\in\mathbb{Z}|f_{f^m(X)}:f^m(X)\to f^{m+1}(X),\\forall m\geq k\}
$$
Furthermore, for any $R$-submodule $Y$ of $X$, we have that $\eta_{f(Y)}\leq \eta_{f(X)}$. 
\end{Lem}

\begin{Def}\cite[Definition 2.6]{LM}
The (relative) $F$-Igusa Todorov function $\phd_F(M):\Ob(\modcat{A})\to\mathbb{N}$ is defined, by using Fitting’s Lemma, as follows
$$
\phd_F(M):=\eta_{\Omega_F}(\langle M\rangle).
$$
\end{Def}

We summarise the properties of relative Igusa-Todorov function below \cite[Propositions 2.11 and 2.12]{LM}.

\begin{Lem}
Let $A$ be an artin $R$-algebra and $M$ and $N$ are $A$-modules. Then the following statements are true.
\begin{enumerate}
\item[(a)] $\phd_F(M)=\pd_F(M)$ if $\pd_F(M)<\infty$.
\item[(b)] $\phd_F(M)=0$ if $M$ is indecomposable and $\pd_F(M)=\infty$.
\item[(c)]  $\phd_F(M)\leq \phd_F(M\oplus N)=0$.
\item[(d)] $\phd_F(M)=\phd_F(N)=0$ if $\add(M)=\add(N)$.
\item[(e)] $\phd_F(M\oplus P)\leq \phd_F(M)=0$ for any $P\in\mathcal{P}(F)$.
\item[(f)] $\phd_F(M)\leq \phd_F(\Omega_F(M))+1$.
\end{enumerate}
\end{Lem}

\begin{Prop}\label{11.1}
Let $A$ be an artin $R$-algebra and $M$ and $N$ are $A$-modules. Then the following conditions are equivalent.
\begin{enumerate}
\item[(a)] $\Ext_F^1(M,-)\simeq \Ext_F^1(N,-)$ 
\item[(b)] $M\oplus P(M)\simeq N\oplus P(N)$ in $\modcat{A}$.
\item[(c)] $M\simeq N$ in $\smod{A}_F$
\item[(d)] $[M]=[N]$ in $K_F(A)$
\end{enumerate}
\end{Prop}
\begin{proof}
(a) and (c) are equivalent.
If $\Ext_F^1(M,-)\simeq \Ext_F^1(N,-)$, then $\Ext_F^1(M,\tau(-))\simeq \Ext_F^1(N,\tau(-))$.
Therefore, by the relative Auslander-Reiten formula \cite[Lemma 2.2]{ASo1} that $\Ext_F^1(M,\tau(-))\simeq \D\stHom_F(-,M)$. Consequently, $\stHom_F(-,M)\simeq \stHom_F(-,N)$. By Yoneda Lemma, we have an isomorphism $M\simeq N$ in $\smod{A}_F$.
We can prove other statements by the similar method in \cite[Proposition 3.1]{FLH}.
\end{proof}

For an artin algebra $A$, denote by $\mathcal {E}(A)$ the quotient of the free abelian group generated by the iso-classes $[\Ext^1_F(-,-)]$ in $\mathcal{C}_A$, for all $M\in\modcat{A}$ modulo the relations
$$
[\Ext^1_F(N,-)]=[\Ext^1_F(X,-)]+[\Ext^1_F(Y,-)] \quad \text{if}\quad N\simeq X\oplus Y
$$

The relative syzygy functor $\Omega_F:\modcat{A}_F\to \modcat{A}_F$ gives rise to a group homomorphism $\Omega_F:\mathcal {E}(A)\to \mathcal {E}(A)$ given by $\Omega_F[\Ext^1_F(N,-)]:=[\Ext^1_F(\Omega_F(N),-)]$. 

\begin{Theo}\label{11.2}
  Let $A$ be an artin $R$-algebra and $M$ an $A$-module. Then the following statements hold.
\begin{enumerate}
\item[(a)] The map $\epsilon:K_F(A)\to \mathcal{E}_F(A)$ given by $\epsilon([M])=[\Ext^1_F(M,-)]$ is an isomorphism of abelian group and the following diagram is commutative

$$\xymatrix@C=8mm@M=3mm{
K_F(A) \ar[r]^{\epsilon}\ar[d]^{\Omega}&\mathcal{E}_F(A)\ar[d]^{\Omega}\\
K_F(A) \ar[r]^{\epsilon} &\mathcal{E}_F(A)
}
$$
\item[(b)] $\epsilon(\Omega^n_F[M])=[\Ext^{n+1}_F(M,-)]$ for any $n\in\mathbb{N}$.
\end{enumerate}
\end{Theo}

\begin{proof}
(a) Since the bi-functor $\Ext^1_F(-,-)$ commutes with finite direct sums and takes, in
the first variable, objects of $\mathcal{P}(F)$ to zero, we get that the map 
$\epsilon: K_F(A)\to\mathcal{E}_F(A)$ is well
defined, surjective and it is also a morphism of abelian groups. Therefore, in order to prove
that $\epsilon$ is an isomorphism of abelian groups, it is enough to see that $\Ker (\epsilon) = 0$.
Let $x\in\Ker (\epsilon)$. Then $x=[M]-[N]$ for some $M,N\in\modcat{A}$. Thus, $\epsilon([M])=\epsilon([N])$. By Proposition \ref{11.1}, it follows that $[M]=[N]$ in $K_F(A)$, then $x=0$. For any $M\in\modcat{A}$, we have 
$$
\epsilon(\Omega_F([M]))=[\Ext^1_F(\Omega_F(M),-)]=\Omega_F(\epsilon([M])).
$$
(b) Let $M\in\modcat{A}$ and $n\in\mathbb{N}$. Then we have 
$$
\epsilon(\Omega^n_F([M]))=[\Ext^1_F(\Omega^n_F(M),-)]=[\Ext^{n+1}_F(M,-)].
$$
\end{proof}

We have the following corollary by Theorem \ref{11.2}.

\begin{Koro}
     Let $A$ be an artin $R$-algebra and $M,N$ $A$-modules. Then the following conditions are equivalent.
\begin{enumerate}
\item[(a)] $[\Omega^n_FM=[\Omega^n_FN]$
\item[(b)] $\Ext^t_B(M,-)\ncong\Ext^t_B(N,-)$ for $t\geq n+1$
\item[(c)] $\Ext^{n+1}_B(M,-)\simeq\Ext^{n+1}_B(N,-)$.

\end{enumerate}
     
\end{Koro}

\begin{Def}
Let $A$ be an artin algebra and $d$ be a positive integer and $M$ an $A$-module. A pair $(X,Y)$ of objects in $\add(M)$ is called a relative $d$-Division of $M$ if the following three conditions hold
\begin{enumerate}
\item[(a)] $\add(X)\cap\add(Y)=0$
\item[(b)] $\Ext^d_F(M,-)\ncong\Ext^d_F(N,-)$
\item[(c)] $\Ext^{d+1}_F(M,-)\simeq\Ext^{d+1}_F(N,-)$.
\end{enumerate}
    
\end{Def}
   
We give a characterization of $\phd_F(M)$ in terms of the bi-functor $\Ext^{1}_F(-,-)$ in the following result.
\begin{Theo}\label{ext-ph}
   Let $A$ be an artin $R$-algebra and $M$ a $A$-module. Then 
   $$
\phd_F(M)=\max(\{d\in\mathbb{Z}|\text{there is a relative d-Division of M}\}\cup {0})
   $$
\end{Theo}
\begin{proof}
Let $n=\phd_F(M)>0$ and $M=\oplus^n_{i=1}M_i$ be a direct sum decomposition of $M$, where $M_i$ is indecomposable and $M_i\ncong M_j$ for $i\neq j$. 
$$
\phd_F(M)=\min\{m\in\mathbb{Z},  \rank\Omega^j_F(\langle M\rangle)=\rank\Omega^m_F(\langle M\rangle), \forall j\geq m\},
$$
it follows that the existence of natural numbers $\alpha_1,\cdots,\alpha_t$ and a partition $\{1,2,\cdots,t\}=I\uplus J$ such that $\Sigma_{i\in I}\alpha_i[\Omega_F^nM_i]=\Sigma_{j\in J}\alpha_j[\Omega_F^nM_j]$, and $\Sigma_{i\in I}\alpha_i[\Omega_F^{n-1}M_i]\neq\Sigma_{j\in J}\alpha_j[\Omega_F^{n-1}M_j]$.

Therefore, the pair $(X,Y)$ with $X=\oplus_{i\in I}M^{\alpha_i}_i$ and $Y=\oplus_{j\in J}M^{\alpha_j}_j$  is a relative $d$-Division of $M$. Moreover, by the fact $\rank\Omega^j_F(\langle M\rangle)>\rank\Omega^n_F(\langle M\rangle)$ for $j=0,1,\cdots,n-1$, completes the proof.

\end{proof}

\section{Relative derived equivalences and relative
Igusa-Todorov dimensions}

Let $\mathcal{A}$ be an additive  category. A complex $\cpx{X}$ over ${\cal A}$ is a sequences $d_X^i$ between objects $X^i$ in ${\cal A}$: $\cdots\lra X^{i-1}\lraf{d_X^{i-1}}X^i\lraf{d_X^i}X^{i+1}\lraf{d_X^{i+1}}\cdots$ such that $d_X^id_X^{i+1}=0$ for all $i\in\mathbb{Z}$.  The category of complexes over ${\cal A}$, in which morphisms are chain maps, is denoted by $\C{\cal A}$, and the corresponding homotopy category is denoted by $\Kk{\cal A}$. When ${\cal A}$ is an abelian category, we write $\Dd{\cal A}$ for the derived category of ${\cal A}$.
We also write $\Kb{\cal A}$, $\Kf{\cal A}$ and $\Kz{\cal A}$ for the full subcategories of $\Kk{\cal A}$ consisting of complexes isomorphic to bounded complexes, complexes bounded above, and complexes bounded below, respectively.  Similarly, for $*\in\{b, -, +\}$, we have $\Dd[*]{\cal A}$. Moreover, for integers $m\leq n$ and for a collection of objects ${\cal  X}$, we write $\Dd[{[m, n]}]{\cal X}$ for the full subcategory of $\Dd{\cal A}$ consisting of complexes $\cpx{X}$ isomorphic in $\Dd{\mathcal{A}}$ to complexes  with terms in ${\cal X}$ of the form
$$0\lra X^m\lra\cdots\lra X^n\lra 0.$$
For each complex $\cpx{X}$ over $\mathcal{A}$, its $i$th cohomology is denoted by $H^i(\cpx{X})$.

Let $\cpx{X}$ be a complex. Then we have the following truncations
$$
\tau_{\leq n}: \cdots\lraf{d_X^{n-3}}X^{n-2}\lraf{d_X^{n-2}}X^{n-1}\lraf{d_X^{n-1}} X^n\lra 0\lra \cdots
$$
and 
$$
\tau_{\geq n}: \cdots\lra 0\lra X^{n}\lraf{d_X^{n}}X^{n+1}\lraf{d_X^{n+1}} X^{n+2}\lra \cdots
$$

\begin{Lem}$\rm\cite[Lemma\,3.7]{P}$\label{7.6} Let $m,t,d \in \mathbb{N}$, $X^{\bullet}, Y^{\bullet} \in
\Kb{A\text{-}mod}$. Assume that $X^{i}=0$ for $i<m$, $Y^{j}=0$
for $j>t$, and $\Ext_{F}^{l}(X^{i},Y^{j})=0$ for all $i,j\in
\mathbb{N}$ and $l\geq d$. Then
$\Hom_{\DFb{A}}(X^{\bullet},Y^{\bullet}[l])=0$ for $l\geq
d+t-m$.
\end{Lem}

\begin{Lem}\cite[Lemma 4.2]{FLH}\label{12.2}
Let $k\geq 0$ and $l>0$, and let $\cpx{X},\cpx{Y}\in \Cc[{[-k, 0]}]{A}$ with $X^i,Y^i\in\pmodcat{A}$ for $-k+1\leq i\leq0$. If
$\Hom_{\Db{A}}(\cpx{X},(-[t]))|_{\modcat{A}}\simeq\Hom_{\Db{A
}}(\cpx{Y},(-[t]))|_{\modcat{A}} \quad\text{for}\quad t\geq k+l$, then
$$
\Ext_A^{t}(X^{-k},-)\simeq\Ext_A^{t}(Y^{-k},-),
$$
for $t\geq k+l$. 
\end{Lem}
\begin{proof}
For the convenient to the reader, we give the proof here. Assume that
$\Hom_{\Db{A}}(\cpx{X},(-[t]))|_{\modcat{A}}\simeq\Hom_{\Db{A
}}(\cpx{Y},(-[t]))|_{\modcat{A}} \quad\text{for}\quad t\geq k+l$. Let $C$ be an $A$-module. 
For $t\geq k+l$, applying the functor $\Hom(-,C[t])$ to the following triangles, 
$$
\begin{aligned}
X^{0}[-1]\to \cpx{X}[-1]\to \tau_{\leq-1}(\cpx{X})[-1]\to X^{0}\\
Y^{0}[-1]\to \cpx{Y}[-1]\to \tau_{\leq -1}(\cpx{Y})[-1]\to Y^{0},
\end{aligned}
$$
we get the following exact sequences
$$
\begin{aligned}
\Hom_{\Db{A}}(X^0,C[t])\to\Hom_{\Db{A}}(\tau_{\leq -1}(\cpx{X})[-1],C[t])\to \Hom_{\Db{A}}(\cpx{X}[-1],C[t])\to\Hom_{\Db{A}}(X^0[-1],C[t])\\
\Hom_{\Db{A}}(Y^0,C[t])\to\Hom_{\Db{A}}(\tau_{\leq -1}(\cpx{Y})[-1],C[t])\to \Hom_{\Db{A}}(\cpx{Y}[-1],C[t])\to\Hom_{\Db{A}}(Y^0[-1],C[t])
\end{aligned}
$$
For $t\geq k+l$, $\Hom_{\Db{A}}(X^0,C[t])\simeq \Hom_{\Db{A}}(X^0[-1],C[t])=0$ and $\Hom_{\Db{A}}(Y^0,C[t])\simeq \Hom_{\Db{A}}(Y^0[-1],C[t])=0$
since 
$X^0$ and $Y^0$ are projective. It follows that $\Hom_{\Db{A}}(\tau_{\leq -1}(\cpx{X})[-1],C[t])\simeq \Hom_{\Db{A}}(\cpx{X}[-1],C[t])$ and $\Hom_{\Db{A}}(\tau_{\leq -1}(\cpx{Y})[-1],C[t])\simeq \Hom_{\Db{A}}(\cpx{Y}[-1],C[t])$. Hence, $\Hom_{\Db{A}}(\tau_{\leq -1}(\cpx{X})[-1],C[t])\simeq\Hom_{\Db{A}}(\tau_{\leq -1}(\cpx{Y})[-1],C[t])$.
When $1<i\leq k-1$, applying the functor $\Hom(-,C[t])$ to the following triangles
$$
\begin{aligned}
X^{-i}\to \tau_{\leq -i}(\cpx{X})[-i]\to \tau_{\leq-i-1}(\cpx{X})[-i]\to X^{-i}[1]\\
Y^{-i}\to \tau_{\leq -i}(\cpx{Y}[-i]\to \tau_{\leq -i-1}(\cpx{Y})[-i]\to Y^{-i}[1]
\end{aligned}
$$
We get
$$
\begin{aligned}
\Hom_{\Db{A}}(\tau_{\leq -i}(\cpx{X})[-i],C[t])\simeq \Hom_{\Db{A}}(\tau_{\leq-i-1}(\cpx{X})[-i],C[t])\\
\Hom_{\Db{A}}(\tau_{\leq -i}(\cpx{Y})[-i],C[t])\simeq \Hom_{\Db{A}}(\tau_{\leq-i-1}(\cpx{Y})[-i],C[t]).
\end{aligned}
$$
since that $X^{-i},Y^{-i}$ projective for $t\geq k+l$. Hence, $\Hom_{\Db{A}}\tau_{\leq -i}(\cpx{X})[-i],C[t])\simeq\Hom_{\Db{A}}(\tau_{\leq -i}(\cpx{Y})[-i],C[t])$.
Finally, we obtain $\Ext_A^{t}(X^{-k},-)\simeq\Ext_A^{t}(Y^{-k},-)$ for $t\geq k+l$. 

\end{proof}

\begin{Lem}\cite[Corollary 4.5]{FLH}\label{12.3}
Let $k$ and $t$ be integers such that $t> k> 0$, and let $\cpx{X},\cpx{Y}\in \Cc[{[-k, 0]}]{A}$ with $X^i,Y^i\in\pmodcat{A}$ for $-k+1\leq i\leq0$. If
$\Ext_A^{t-k}(X^{-k},-)\simeq\Ext_A^{t-k}(Y^{-k},-)$, then 
$$
\Hom_{\Db{A}}(\cpx{X},(-)[t])|_{\Dd[{[-k, 0]}]{A}}\simeq\Hom_{\Db{A
}}(\cpx{Y},(-)[t])|_{\Dd[{[-k, 0]}]{A}}  
$$
\end{Lem}
\begin{proof}
Here we give an alternative proof which simplifies the proof of \cite[Corollary 4.5]{FLH}.
Assume that
$\Ext_A^{t-k}(X^{-k},-)\simeq\Ext_A^{t-k}(Y^{-k},-)$, we claim
$\Hom_{\Db{A}}(X^{-k},(-)[t-k])|_{\Dd[{[-k, 0]}]{A}}\simeq\Hom_{\Db{A
}}(Y^{-k},(-)[t-k])|_{\Dd[{[-k, 0]}]{A}}  
$. 
Suppose that $\cpx{Z}\in \Dd[{[-k, 0]}]{A}$ with $Z^i$ injective for $-k\leq i\leq -1$. Then we have a triangle 
$Z^0\to \cpx{Z}\to\tau_{\leq -1}(\cpx{Z})\to Z^0[1]$ with $\tau_{\leq -1}(\cpx{Z})\in\Kb{\imodcat{A}}$.
There are two exact sequences
$$
\begin{aligned}
\cdots\to\Hom_{\Db{A}}(X^{-k},\tau_{\leq -1}
(\cpx{Z})[l-1])\to \Hom_{\Db{A}}(X^{-k},Z^0[l])\to\Hom_{\Db{A}}(X^{-k},\cpx{Z}[l])\\\to\Hom_{\Db{A}}(X^{-k},\tau_{\leq -1}
(\cpx{Z})[l])\to\cdots\\
\cdots\to\Hom_{\Db{A}}(Y^{-k},\tau_{\leq -1}
(\cpx{Z})[l-1])\to \Hom_{\Db{A}}(Y^{-k},Z^0[l])\to\Hom_{\Db{A}}(Y^{-k},\cpx{Z}[l])\\\to\Hom_{\Db{A}}(Y^{-k},\tau_{\leq -1}
(\cpx{Z})[l])\to\cdots
\end{aligned}
$$ for any integer $l\geq 1$.
Note that $\tau_{\leq -1}(\cpx{Z})\in\Kb{\imodcat{A}}$, and $l\geq 1$,
$\Hom_{\Db{A}}(X^{-k},\tau_{\leq -1}
(\cpx{Z})[l-1])=\Hom_{\Kb{A}}(X^{-k},\tau_{\leq -1}
(\cpx{Z})[l-1])=0$, $\Hom_{\Db{A}}(X^{-k},\tau_{\leq -1}
(\cpx{Z})[l])=\Hom_{\Kb{A}}(X^{-k},\tau_{\leq -1}
(\cpx{Z})[l])=0$. We have $ \Hom_{\Db{A}}(X^{-k},Z^0[l])\simeq\Hom_{\Db{A}}(X^{-k},\cpx{Z}[l])
$. Similarly, we get $ \Hom_{\Db{A}}(Y^{-k},Z^0[l])\simeq\Hom_{\Db{A}}(Y^{-k},\cpx{Z}[l])
$. Hence $\Hom_{\Db{A}}(X^{-k},(-)[t-k])|_{\Dd[{[-k, 0]}]{A}}\simeq\Hom_{\Db{A
}}(Y^{-k},(-)[t-k])|_{\Dd[{[-k, 0]}]{A}}  
$. 

On the other hand, there exists a triangle 
$\tau_{\geq -k+1}(\cpx{Z})\to\cpx{Z}\to Z^{-k}[k]\to\tau_{\geq -k+1}(\cpx{Z})[1] $ with $\tau_{\geq -k+1}(\cpx{Z})\in\Kb{\pmodcat{A}}$.
It follows that there are two exact sequences
$$
\begin{aligned}
\cdots\to\Hom_{\Db{A}}(\tau_{\geq -k+1}(\cpx{X})[1],V[l])\to \Hom_{\Db{A}}(X^{-k}[k],V[l])\to\Hom_{\Db{A}}(\cpx{X},V[l])\\\to\Hom_{\Db{A}}(\tau_{\geq -k+1}(\cpx{X}),V[l])\to\cdots\\
\cdots\to\Hom_{\Db{A}}(\tau_{\geq -k+1}(\cpx{Y})[1],V[l])\to \Hom_{\Db{A}}(Y^{-k}[k],V[l])\to\Hom_{\Db{A}}(\cpx{Y},V[l])\\\to\Hom_{\Db{A}}(\tau_{\geq -k+1}(\cpx{Y}),V[l])\to\cdots
\end{aligned}
$$ for any integer $l\geq 1$.
Note that $\tau_{\geq -k+1}(\cpx{Z})\in\Kb{\pmodcat{A}}$, and $l\geq 1$,
$\Hom_{\Db{A}}(\tau_{\geq -k+1}(\cpx{X})[1],V[l])=\Hom_{\Kb{A}}(\tau_{\geq -k+1}(\cpx{X})[1],V[l])=0$, $\Hom_{\Db{A}}(\tau_{\geq -k+1}(\cpx{X}),V[l])=\Hom_{\Kb{A}}(\tau_{\geq -k+1}(\cpx{X}),V[l])=0$. We have $  \Hom_{\Db{A}}(X^{-k}[k],V[l])\simeq\Hom_{\Db{A}}(\cpx{X},V[l])$. Similarly, $  \Hom_{\Db{A}}(Y^{-k}[k],V[l])\simeq\Hom_{\Db{A}}(\cpx{Y},V[l])$. 
Then $\Hom_{\Db{A}}(\cpx{X},(-)[t])|_{\Dd[{[-k, 0]}]{A}}\simeq\Hom_{\Db{A}}(\cpx{Y},(-)[t])|_{\Dd[{[-k, 0]}]{A}} 
$.
\end{proof}

Let $T^{\bullet}$ be an $F$-tilting complex associated to $L: \DFb{\Modcat{A}}\lra \Db{\Modcat{B}}$ of
the form:
$$
\cdots\rightarrow0\rightarrow T^{-n}\rightarrow
T^{-n+1}\rightarrow\cdots \rightarrow T^{-1}\rightarrow
T^{0}\rightarrow 0 \rightarrow\cdots.
$$ 

\begin{Lem} Let $X$ be a $A$-module. Then we have the follow statements hold.

(1) $H^{i}(L(X))=0$ for $i>n$ or $i<0$, and in $\Db{B}$, $L(X)$ is
isomorphic to a complex $\cpx{\bar{T}_{X}}$ of the following form:
$$
\cdots\ra 0\ra \bar{T}_{X}^{0}\ra \bar{T}_{X}^{1}\ra \cdots
\ra \bar{T}_{X}^{n-1}\ra \bar{T}_{X}^{n}\ra 0\ra \cdots
$$
where $\bar{T}_{X}^{i}$ is an projective $B$-modules for each $1\leq i\leq n$.

(2)
$H^{i}(L(X))=0$ for $i>n$ or $i<0$, and in $\Db{B}$, $L(X)$ is
isomorphic to a complex $\cpx{\bar{J}_{X}}$ of the following form:
$$
\cdots\ra 0\ra \bar{J}_{X}^{0}\ra \bar{J}_{X}^{1}\ra \cdots
\ra \bar{J}_{X}^{n-1}\ra \bar{J}_{X}^{n}\ra 0\ra \cdots
$$
where $\bar{T}_{X}^{i}$ is an injective $B$-modules for each $0\leq i\leq n-1$.
\end{Lem}
 
\begin{proof} 
(1) It is proved in \cite[Lemma 6.4]{P1}.

(2) Suppose that $L(X)$ is
isomorphic to a complex $\cpx{\bar{J}_{X}}$ in $\Db{B}$.
Considering the homological group of $\cpx{\bar{J}_{X}}$, we
have
$$
H^{i}(L(X))\simeq
\Hom_{\Kb{B}}(B,L(X)[i])\simeq
\Hom_{\Db{B}}(B,L(X)[i])\simeq
\Hom_{\DFb{A}}(T^{\bullet},X[i]).
$$
Then we have
$\Hom_{\DFb{A}}(T^{\bullet},X[i])\simeq
\Hom_{\Kb{A}}(T^{\bullet},X[i])$. Then
$H^{i}(L(X))=0$ for $i>n$ or $i<0$ by transferring
shifts. 
If $H^{i}(\cpx{\bar{J}_{X}})=0$ for
$i<0$, then the following sequence 
$$0\to\cdots\to\bar{J}_{X}^{-1}\to\bar{J}_{X}^{-1}\to \Ima(d^{-1})\to 0 \to\cdots$$ is a split
exact sequence, and 
$$0\to \Coker(d_J^{-1})\to
\bar{J}^{1}\to \cdots \to \bar{J}^{s}\to0 \to\cdots$$ is an
exact sequence, where $Coker(d_J^{-1})$ is injective.
If $H^{i}(\cpx{\bar{T_J}})=0$ for
$i>n$, then the following sequence $$0\to \Ima(d_J^{n})\to
\bar{J_x}^{n+1}\to \cdots \to \bar{J_X}^{s}\to0 \to\cdots$$ is an
exact sequence. Consequently, the complex $\cpx{\bar{J}_{X}}$ is isomorphic
in $\Db{B}$ to a complex of the form
$$
0\ra \bar{J}_{X}^{0}\ra \bar{J}_{X}^{1}\ra \cdots
\ra \bar{J}_{X}^{n-1}\ra \bar{J}_{X}^{n}\ra 0 \quad\text{in}\quad
\Db{B},
$$
where $\bar{J}_{X}^{i}$ is an injective $B$-modules for each $0\leq i\leq n-1$.
This completes the proof. 
\end{proof}

\begin{Lem}\label{res-L}
 Let $H:\Db{B}\ra \DFb{A}$ be a quasi-inverse of $L:\DFb{A}\ra \Db{B}$ and $Y\in\modcat{B}$. Then 
\begin{enumerate}
\item[(1)]
$H(Y)$ is isomorphic to a complex $\cpx{T_{Y}}$ of the following form
$$
\cdots\ra 0\ra  T_{Y}^{-n-2}\ra T_{Y}^{-n-1}\ra \cdots \ra
T_{Y}^{-1}\ra  T_{Y}^0\ra 0\ra \cdots,
$$
where $T_{Y}^{i}$ is an $F$-projective $A$-module for each $-n-1\leq i\leq 0$.
\item [(2)]
$H(Y)$ is isomorphic to a complex $\cpx{I_{Y}}$ of the following form
$$
\cdots\ra 0\ra  I_{Y}^{-n-2}\ra I_{Y}^{-n-1}\ra \cdots \ra
I_{Y}^{-1}\ra  I_{Y}^0\ra 0\ra \cdots,
$$
where $I_{Y}^{i}$ is an $F$-injective $A$-module for each $-n-2\leq i\leq -1$.
\end{enumerate}
\end{Lem}
\begin{proof}
(1) The proof is given in \cite[Lemma 6.3]{P1}.

(2) By Lemma \ref{relative-til},  the complex $\cpx{\bar{T}}\in\Kb{\pmodcat{B}}$ associated to $H$ of the form
$0\lra \bar{T}^0\lra \bar{T}^1\lra\cdots\lra\bar{T}^n\lra 0$.
First, we have the following isomorphisms
$$
H^{i}(H(Y))\simeq \Hom_{\DFb{A}}(A,H(Y)[i])\simeq
\Hom_{\Db{B}}(\bar{T}_{A}^{\bullet},Y[i]) \simeq
\Hom_{\K{B}}(\bar{T}_{A}^{\bullet},Y[i]),
$$
where $L(A)\simeq \bar{T}_{A}^{\bullet}$ has the same
form as $\bar{T}^{\bullet}$. Then $H^{i}(H(Y))=0$ for $i>0$ or
$i<-n$. Assume that $H(Y)$ is isomorphic to a bounded above
$F$-injective complex:
$$
\cdots\ra T_{Y}^{-n-1}\ra T_{Y}^{-n}\ra \cdots\ra T_{Y}^{-1}\ra
T_{Y}^{0}\ra\cdots\ra T_{Y}^{s}\ra0
$$
Then we get two acyclic complexes
$$
\cdots\ra T_{Y}^{-n-1}\ra \Ima(d_{T_{Y}}^{-n-1})\ra 0\ra \cdots
$$
and
$$
0\ra \Ima(d_{T_{Y}}^{0})\ra T_{Y}^{1}\ra \cdots.
$$
For any $P\in \mathcal {P}(F)$, we have
$$
H^{i}(P,H(Y))\simeq \Hom_{\DFb{A}}(P,H(Y)[i])\simeq
\Hom_{\Db{B}}(\bar{T}_{P}{^{\bullet}},Y[i]) \simeq
H^{i}(\Hom_{B}(\bar{T}_{P}{^{\bullet}},Y)) =0
$$
for $i>0$ or $i<-n$, where $L(P)\simeq \bar{T}_{P}{^{\bullet}}$ has
the same form as $\bar{T}^{\bullet}$. Therefore, we get the
following acyclic complexes:
$$
\cdots\ra \Hom(P,T_{Y}^{-n-1})\ra \Ima\Hom(P,d_{T_{Y}}^{-n-1})\ra 0
$$
and
$$
0\ra \Ima(\Hom(P,d_{T_{Y}}^{0}))\ra \Hom(P,T_{Y}^{1})\ra \cdots.
$$
We conclude that there exist two $F$-acyclic complexes:
$$
\cdots\ra T_{Y}^{-n-2}\ra \Ima(d_{T_{Y}}^{-n-2})\ra 0\ra \cdots,
$$
and
$$
0\ra \Ima(d_{T_{Y}}^{0})\ra T_{Y}^{1}\ra \cdots.
$$
Then the complex $H(Y)$ is isomorphic in $\DFb{A}$ to a
complex $\cpx{T_{Y}}$ of the following form:
$$
\cdots\ra 0\ra \Ima(d_{T_{Y}}^{-n-2})\ra T_{Y}^{-n-1}\ra \cdots \ra
T_{Y}^{-1}\ra \Ker(d_{T_{Y}}^{0})\ra 0\ra \cdots,
$$
where $\Ima(d_{T_{Y}}^{-n-2})$ and $T_{Y}^{-i}$($i=1,\cdots, n-1$) are
$F$-injective $A$-modules. 
\end{proof}

\begin{Lem}\label{res-H}
 Let $L:\DFb{A}\ra \Db{B}$ be a
relative derived equivalence such that $B=\End(\cpx{T})$. Let $X,Y\in\modcat{A}$, and $H: \Db{B}\to\DFb{A}$ its quasi-inverse. If $\Ext^t_B(M,-)\simeq\Ext^t_B(N,-)$ for all $t\geq l>0$, then 
$$
\Hom_{\Db{B}}(M,L(-[t]))|_{\modcat{A}}\simeq\Hom_{\Db{B}}(N,L(-[t]))|_{\modcat{A}} \quad\text{for}\quad t\geq n+l.
$$   
\end{Lem}
\begin{proof}
The proof is similar to that of \cite[Lemma 4.9]{FLH}. 
Let $V\in\modcat{A}$. By Lemma 
\ref{res-H}(2), in $\Db{B}$, $L(V)$ is
isomorphic to a complex $\cpx{J}_{V}$ of the following form:
$$
\cdots\ra 0\ra J_{V}^{0}\ra J_{V}^{1}\ra \cdots
\ra J_{V}^{n-1}\ra J_{v}^{n}\ra 0\ra \cdots
$$
where $J_{V}^{i}$ is an injective $B$-modules for each $0\leq i\leq n-1$. Then there exists a triangle 
$\tau_{\leq n-1}(\cpx{J}_{V})[t-1]\to J_{V}^{n}[-n+t]\to \cpx{J}_{V}[t]\to \tau_{\leq n-1}(\cpx{J}_{V})[t]$.
Applying the funcors $\Hom_{\Db{B}}(M,-)$ and $\Hom_{\Db{B}}(N,-)$ to the above triangle, we get two exact sequences
$$
\begin{aligned}
\cdots\to\Hom_{\Db{A}}(M,\tau_{\leq n-1}(\cpx{J}_{V})[t-1])\to \Hom_{\Db{A}}(M,J_{V}^{n}[-n+t])\to\Hom_{\Db{A}}(M,\cpx{J}_{V}[t])\\\to\Hom_{\Db{A}}(\tau_{\geq -k+1}(\cpx{J}),V[t])\to\cdots\\
\cdots\to\Hom_{\Db{A}}(N,\tau_{\leq n-1}(\cpx{J}_{V})[t-1])\to \Hom_{\Db{A}}(N,J_{V}^{n}[-n+t])\to\Hom_{\Db{A}}(N,\cpx{J}_{V}[t])\\\to\Hom_{\Db{A}}(N,\tau_{\geq -k+1}(\cpx{J}),V[t])\to\cdots\\
\end{aligned}
$$ for any integer $t\geq n+l$.

Note that $\tau_{\leq n-1}(\cpx{J}_{V})\in\Kb{\imodcat{A}}$, and $t\geq n+l\geq 1$,
$\Hom_{\Db{A}}(M,\tau_{\leq n-1}(\cpx{J}_{V})[t-1])=\Hom_{\Kb{A}}(M,\tau_{\leq n-1}(\cpx{J}_{V})[t-1])=0$, $\Hom_{\Db{A}}(M,\tau_{\leq n-1}(\cpx{J}_{V})[t])=\Hom_{\Kb{A}}(M,\tau_{\leq n-1}(\cpx{J}_{V})[t])=0$. We have $\Hom_{\Db{A}}(M,J_{V}^{n}[-n+t])\simeq\Hom_{\Db{A}}(M,\cpx{J}_{V}[t])$. Similarly, $\Hom_{\Db{A}}(N,J_{V}^{n}[-n+t])\simeq\Hom_{\Db{A}}(N,\cpx{J}_{V}[t])$. Hence, $\Hom_{\Db{B}}(M,L(-[t]))|_{\modcat{A}}\simeq\Hom_{\Db{B}}(N,L(-[t]))|_{\modcat{A}} \quad\text{for}\quad t\geq n+l.$

\end{proof}

\begin{Theo}\label{Main-result} Let $L:\DFb{A}\ra \Db{B}$ be a
relative derived equivalence. Suppose $T^{\bullet}$ is the relative
tilting complex associated to $L$. Then we have
$\phd_{F}(A)-t(T^{\bullet})\leq \phd(B)\leq
\phd_{F}(A)+t(T^{\bullet})+2$.

\end{Theo}
\begin{proof}
Assumet that $\phd_F(A)=r<\infty$. If $\phd(B)\leq n$, then $\phd(B)\leq \phd_{F}(A)+n$. Suppose that there exists $X\in\modcat{B}$ such that $\phd(X)=d>n$. It follows that $X=\bar{M}\oplus \bar{N}$, and there exist $M\in\add(\bar{M})$ and $N\in\add(\bar{N})$ such that 
\begin{equation}
\left\{
\begin{array}{cc}
\Ext^d_B(M,-)\ncong\Ext^d_B(N,-),\\ 
\Ext^t_B(M,-)\simeq\Ext^t_B(N,-), & \text{for}\ \ t\geq d+1.  
\end{array} \right. 
\end{equation}
By Lemma \ref{res-H}, the second equation 
\begin{equation}
\left\{
\begin{array}{cc}
\Hom_{\Db{B}}(M,-[d])|_{\modcat{B}}\ncong\Hom_{\Db{B}}(N,-[d])|_{\modcat{B}},\\ 
\Hom_{\Db{B}}(M,L(-[t]))|_{\modcat{A}}\simeq\Hom_{\Db{B}}(N,L(-[t]))|_{\modcat{A}} 
\end{array} \right.  
\end{equation}
for all $t\geq n+d+1$.
After applying the equivalence $H$, we get
\begin{equation}
\left\{
\begin{array}{cc}
\Hom_{\DFb{A}}(H(M),H(-[d]))\ncong\Hom_{\DFb{A}}(H(N),H(-[d])),\\ 
\Hom_{\DFb{A}}(H(M),(-[t]))|_{\modcat{A}}\simeq\Hom_{\DFb{A}}(H(N),(-[t]))|_{\modcat{A}}
\end{array} \right. 
\end{equation}
for all $t\geq n+d+1$.
By Lemma \ref{res-L}, $H(M)$ and $H(N)$ have the following forms
$$
\cdots\ra 0\ra  T_{M}^{-n-2}\ra T_{M}^{-n-1}\ra \cdots \ra
T_{M}^{-1}\ra  T_{M}^0\ra 0\ra \cdots,
$$
and
$$
\cdots\ra 0\ra  T_{N}^{-n-2}\ra T_{N}^{-n-1}\ra \cdots \ra
T_{N}^{-1}\ra  T_{N}^0\ra 0\ra \cdots,
$$
By Lemma \ref{12.2}, the second term of Eq. (3) gives us that
\begin{equation}
\Ext^t_F(T^{-n-2}_M,-)\simeq\Ext^t_F(T^{-n-2}_N,-), 
\end{equation}
for $t\geq n+d+1$. 
By Lemma \ref{12.3}, the first term of Eq. (3) is euivalent to
\begin{equation}
\Ext^{d-n-2}_F(T^{-n-2}_M,-)\ncong\Ext^{d-n-2}_F(T^{-n-2}_N,-). 
\end{equation}

In particular, $T^{-n-2}_M\ncong T^{-n-2}_N$. We can assume that $\add (T^{-n-2}_M)$ and
$\add (T^{-n-2}_N)$ have trivial intersection because, otherwise, we can decompose $T^{-n-2}_M\ncong T^{-n-2}_N$
as direct sum of indecomposables and withdraw from each one the common factors. The
modules obtained satisfy (4) and (5), and their additive closure has trivial intersection

Now let $V=T^{-n-2}_M\oplus T^{-n-2}_N $. Then $d-n-2\leq \phd_F(V)\leq d+n$ by Theorem \ref{ext-ph} and Eq.s (4) and (5).
$\phd_F(A)=r<\infty$, we have $d\leq n+2+r$ and hence $\phd(B)\leq n+\phd_F(A)+2.$ Analogously, 
 we can show that  $\phd_F(A)\leq\phd(B)+n.$ 
 This completes the proof.

\end{proof}

\begin{rem}
  If $F=\Ext^1(-,-)$, then we re-obtain the result \cite[Theorem 4.10]{FLH}.  
\end{rem}

\noindent{\bf Acknowledgements.} The second author would like to
acknowledge the Fundamental Research Funds for the Central Universities of Beijing Jiaotong University (2024JBMC001).

\bigskip

\begin{thebibliography}{99}
{\small
\bibitem{ARS}{{\sc  Auslander, M., Reiten, I., Smal{\O}, S. O.}:
Representation Theory of Artin Algebras. Cambridge University
Press, 1995.}
\bibitem{ASo1}{{\sc Auslander, M., Solberg, {\O}.:}
Relative homology and representation theory I, Relative homology and
homologically finite subcategories. \emph{Comm. Algebra} \textbf{21}(1993),
2995-3031.}
\bibitem{ASo2}{{\sc Auslander, M., Solberg, {\O}.:}
Relaive homology and representation theory II, Relative cotilting
theory. \emph{Comm. Algebra}  \textbf{21}(1993), 3033-3079.}

\bibitem{ASo3}{{\sc Auslander, M., Solberg, {\O}.:}
Relative homology and representation theory III, Cotilting modules
and Wedderburn correspondence. \emph{Comm. Algebra} \textbf{21}(1993),
3081-3097.}

\bibitem{ASo4}{{\sc Auslander, M., Solberg, {\O}.:}
Gorenstein algebras and algebras with dominant dimension at least
$2$. \emph{Comm. Algebra}  \textbf{21}(1993), 3897-3934.}

\bibitem{ASo5}{{\sc Auslander, M., Solberg, {\O}.:}
Relative homology.\, Finite-dimensional algebras and related topics
(Ottawa, ON, 1992), 347--359, NATO Adv. Sci. Inst. Ser. C Math.
Phys. Sci., 424, Kluwer Acad. Publ., Dordrecht, 1994.}

\bibitem{Bu1}{{\sc Buan, A.:}
Closed subbifunctors of the extension bifunctor. \emph{J. Algebra}
\textbf{244}(2001), 407-428.}

\bibitem{Bu2}{{\sc Buan, A.:}
Subcategories of the derived category and cotilting complexes. \emph{Colloq. Math.} \textbf{88}(2001), 1-11.}

\bibitem{Buch}{{\sc Buchweitz,  R. O.:}
Maximal Cohen-Macaulay modules and Tate-cohomology over Gorenstein
rings, Hamburg, p. 155 (1987). (unpublished manuscript)}

\bibitem{Ch}{{\sc Chen, X. W.:}
Gorenstein homological algebra of Artin algebras, Postdoctoral
Report, USTC, 2010.}

\bibitem{CYZ}{{\sc Chen, X. W., Ye, Y., Zhang, P.:}
Algebras of dereived dimension of zero. \emph{Comm. Algebra} \textbf{36}(2008),
1-10.}

\bibitem{CBS}{{\sc Cline, E., Parshall, B., Scott L.:}
Derived categories and Morita theory. J. Algebra 104(1986),
397-409.}
\bibitem{DRSS}{{\sc P. Draxler}, {\sc I. Reiten}, {\sc S. O. Smal{\O}} and {\sc {\O}. Solberg},
Exact categories and vector space categories. \emph{Trans. Amer.
Math. Soc.} \textbf{351}(1999), 647-682.}

\bibitem{DS}{{\sc Dugger, D., Shipley, B.:}
K-theory and derived equivalences. \emph{Duke Math. J.} \textbf{124}(2004),
587-617.}

\bibitem{EJ}{{\sc Enochs, E. E., Jenda, O. M. G.:}
Relative Homological Algebra. De Gruyter Expositions in
Math. 30, Walter de Gruyter-Berlin-New York, 2000.}


\bibitem{FLH}
Fernandes S.M., Lanzilotta, M. and Hernandez M., The $\Phi$-dimension: A new homological measure. \emph{Algeb. Represent. Theor.} \textbf{18} (2015), 463-476.


\bibitem{GZ}{{\sc Gao, N., Zhang, P.:}
Gorenstein derived categories. \emph{J. Algebra} \textbf{323}(2010), 2041-2057}.

\bibitem{GM}{{\sc Gelfand,  S. I., Manin, Y.:}
Methods of Homological Algebra. Springer Monographs in
Math. (2nd ed.), Springer-Verlag, 2002.}

\bibitem{Ge}{{\sc Generalov, A. I.:}
Relative homological algebra. Cohomology of categories, posets and
coalgebras, \, in "Handbook of Algebra"\, V.1,\, 611-638.\,
North-Holland, Amsterdam, 1996.}

\bibitem{Ha1}{{\sc Happel, D.:}
Triangulated Categories in the Representation Theory of Finite
Dimensional Algebras. Cambridge University Press, Cambridge. 1988.}

\bibitem{Ha2}{{\sc Happel, D.:}
On the derived category of finite-dimensional algebra. \emph{Comment.
Math. Helv.} \textbf{62}(1987), 339-389.}

\bibitem{Ha3}{{\sc Happel, D.:}
On Gorenstein algebras.\, In: Representation theory of finite groups
and finitedimensional algebras \, (Proc. Conf. at Bielefeld, 1991),
389-404, \, Progress in Math., vol. 95, Birkh$\ddot{a}$user, Basel,
1991.}


\bibitem{Har}{{\sc Harshorne, R.:}
Duality and Residue. Lecture Notes in Math. 20, Springer,
Berlin, 1966.}

\bibitem{Hoch}{{\sc Hochschild, G.:}
Relative homological algebra. \emph{Trans. Amer. Math. Soc.} \textbf{82}(1956),
246-269.}

\bibitem{HK1}{{\sc Hoshino, M., Kato, Y.:}
Tilting complexes defined by idempotents. \emph{Comm. Algebra} \textbf{30}(2002),
83-100.}
\bibitem{HK2}{{\sc Hoshino, M., Kato, Y.:}
Tilting complexes associated with a sequence of idempotents. \emph{J.
Algebra} \textbf{183}(2003), 105-124.}


\bibitem{HP}{{\sc Hu, W., Pan S.:}
Stable functors of derived equivalences and Gorenstein projective modules.
\emph{Math. Nach.} \textbf{290} (2017), no. 10, 1512-1530.}



\bibitem{KS}{{\sc Kashiwara, M., Schapira, P.:}
Sheaves on Manifolds. Grundlehren der mathematischen
Wissenschaften 292\, Berlin: Springer-Verlag, 1990.}

\bibitem{Ka}{{\sc Kato, Y.:}
On derived equivalent coherent rings. \emph{Comm. Algebra} \textbf{30}(2002),
4437-4454.}

\bibitem{Ke1}{{\sc Keller, B.:}
Deriving DG categories. \emph{Ann. Sci. \'{e}cole Norm. Sup.} \textbf{27}
(1994), 63-102.}


\bibitem{KZ}{{\sc Konig, S.,Zimmermann, A.:}
Derived Equivalences for Group Rings. Lecture notes in
math. 1685, Berlin: Springer-Verlag, 1998}

\bibitem{Kr2}{{\sc Krause, H.:}
Localization theory for triangulated categories. \, To appear in the
proceedings of the "Workshop on Triangulated Categories" in Leeds,
2006.}


\bibitem{LM}{{\sc Lanzilotta, M., Mendoza, O.},
Relative Igusa-Todorov functions and relative homological dimensions.
\emph{Algebr. Represent. Theory} \textbf{20} (2017), no. 3, 765-802.}

\bibitem{Mi2}{{\sc Miyashita, Y.:}
Tilting modules associated with a series of idempotent ideals. \emph{J.
Algebra} \textbf{238} (2001), 485–501.}

\bibitem{N}{{\sc Neeman, A.}
Triangulated Categories. Annals of Mathematics Studies 148,
Princeton University Press, \, Princeton and Oxford,\, 2001.}

\bibitem{P1}{{\sc Pan, S. Y.:}
Relative derived equivalences and relative homological dimensions.
\emph{Acta Math. Sin.} \textbf{32} (2016), no. 4, 439-456.}

\bibitem{P}{{\sc Pan, S. Y.:}
Relative quotient triangulated categories. \emph{Algebra Colloq.} \textbf{21} (2014), no. 2, 195-206.}

\bibitem{PX}{{\sc Pan, S. Y., Xi, C. C.:}
Finiteness of finitistic dimension is invariant under derived
equivalences. \emph{J. Algebra} \textbf{322}(2009), 21-24.}


\bibitem{Ri1}{{\sc Rickard, J.:}
Morita theory for derived categories. \emph{J. London Math. Soc.} \textbf{39}
(1989), 436-456.}

\bibitem{Ri2}{{\sc Rickard, J.:}
Derived categories and stable equivalence. \emph{J. Pure Appl. Algebra} \textbf{61}
(1989), 303-317.}

\bibitem{Ri3}{{\sc Rickard, J.:}
Derived equivalences as derived functors. \emph{J. London Math. Soc.}  \textbf{43}
(1991), 37-48.}

\bibitem{RoT}{{\sc Rotman, J.:}
An Introduction to Homological Algebra. Academic Press New
York, San Francisco, London, 1979.}



\bibitem{Ver}{{\sc Verdier, J.:}
Cat\'{e}gories d\'{e}riv\'{e}es, \'{e}tat 0. \, Lecture Notes in
Math. 569(1977), Springer, Berlin, 262-311.}


\bibitem{W}{{\sc Weibel, C. A.:}
An Introduction to Homological Algebra. Cambridge Studies
in Advanced Math. 38,  Cambridge University Press, 1994.}

}
\end{thebibliography}
\end{document}